\documentclass[article]{amsart}

\usepackage{amssymb,amsfonts,amsmath,amsthm}
\usepackage[all,arc]{xy}
\usepackage{enumerate}
\usepackage{mathrsfs}
\usepackage[toc,page]{appendix}
\usepackage[left=3cm, right=3cm, bottom=3cm]{geometry}
\usepackage{tabularx}
\usepackage{url}
\usepackage{color}
\usepackage{tikz-cd}
\usepackage{mathdots}

\newtheorem{thm}{Theorem}[section]

\newtheorem*{thm*}{Theorem}
\newtheorem*{cor*}{Corollary}
\newtheorem*{prop*}{Proposition}
\newtheorem{cor}[thm]{Corollary}
\newtheorem{prop}[thm]{Proposition}
\newtheorem{lem}[thm]{Lemma}

\theoremstyle{definition}

\theoremstyle{remark}

\newtheorem*{idea*}{Idea}

\newcommand{\Spec}{\text{Spec }}

\makeatletter
\let\c@equation\c@thm
\makeatother
\numberwithin{thm}{section}
\numberwithin{equation}{section}

\title[Moduli Problem of Hitchin Pairs over Deligne-Mumford Stack]{Moduli Problem of Hitchin Pairs over Deligne-Mumford Stack}

\author{Hao Sun}

\begin{document}
\maketitle

\renewcommand{\thefootnote}{\fnsymbol{footnote}}
\footnotetext[1]{MSC2010 Class: 14A20}
\footnotetext[2]{Key words: Deligne-Mumford stack, Hitchin pair, moduli problem}

\begin{abstract}
We define the moduli problem of Hitchin pairs over Deligne-Mumford Stack and prove this moduli problem is represented by a separated and locally finitely presented algebraic space, which is considered as the moduli space of Hitchin pairs over Deligne-Mumford stack.
\end{abstract}
\flushbottom
\section{Introduction}
The Hitchin pair was introduced by Hitchin in 1987 \cite{Hit}, which is also known as the Higgs bundle. Numerous mathematicians made great contributions to the construction of the moduli space of Hitchin pairs in the past thirty years. After Hitchin's paper, Niture constructed the moduli space of semistable Hitchin pairs over smooth curves \cite{Nit}, and Simpson constructed the moduli space of semistable Hitchin pairs over smooth projective varieties \cite{Simp2}. Combined with the study of parabolic bundle \cite{MeSe}, people constructed the moduli space of parabolic Higgs bundle \cite{Yoko1}. It is well-known that the construction of the parabolic bundle comes from the bundle over orbifold. If we consider the orbifold as stacks (or as a special type of Deligne-Mumford stack), a natural idea is how to construct the moduli space of Hitchin pairs over Deligne-Mumford stack. In this paper, we partially answer this question by proving the existence of the moduli space of Hitchin pairs over a Deligne-Mumford stack.

Let $\mathcal{X} \rightarrow S$ be a separated, locally finitely-presented morphism from a Deligne-Mumford stack $\mathcal{X}$ to an algebraic space $S$. Let $\mathcal{G}$ be a coherent $\mathcal{O}_{\mathcal{X}}$-module. Olsson and Starr defined the quotient functor $Q(\mathcal{G}/\mathcal{X}/S):=Quot(\mathcal{G},\mathcal{X},S)$ and proved that the quotient functor is represented by an algebraic space which is separated and locally finitely presented over $S$ \cite{OlSt}.

Let $\mathcal{F}$ be a locally finitely-presented quasi-coherent $\mathcal{O}_{\mathcal{X}}$-module in $Q(\mathcal{G}/\mathcal{X}/S)$, and we fix a line bundle (locally free sheaf with rank one) $\mathcal{L}$ over $\mathcal{X}$, which is considered as the twisted bundle. An \emph{$\mathcal{L}$-twisted Higgs field} $\Phi$ on the quasi-coherent sheaf $\mathcal{F}$ is a homomorphism
\begin{align*}
\Phi:\mathcal{F} \rightarrow \mathcal{F} \otimes \mathcal{L}.
\end{align*}
An \emph{$\mathcal{L}$-twisted Hitchin pair} over $\mathcal{X}$ is a pair $(\mathcal{F},\Phi)$, where $\mathcal{F}$ is a locally finitely-presented quasi-coherent sheaf over $\mathcal{X}$ and $\Phi$ is an $\mathcal{L}$-twisted Higgs field. We consider the following moduli problem (functor)
\begin{align*}
\mathcal{M}_{\mathcal{X},\mathcal{L},Higgs}: (\text{Sch/S})^{{\rm op}} \rightarrow \text{Set}
\end{align*}
for $\mathcal{L}$-twisted Hitchin pairs over $\mathcal{X}$ as follows. For each $T \in \text{Sch/S}$, define
\begin{align*}
\mathcal{M}_{\mathcal{X},\mathcal{L},Higgs}(T)=\{ (\mathcal{F}_T,\Phi_T) \text{ } | \text{ }\mathcal{F}_T \in Q(\mathcal{G}/\mathcal{X}/S)(T), \Phi_T: \mathcal{F}_T \rightarrow \mathcal{F}_T \otimes p_{\mathcal{X}}^* \mathcal{L} \} ,
\end{align*}
where $p_{\mathcal{X}}:\mathcal{X} \times_S T \rightarrow \mathcal{X}$ is the natural projection. The main result of this paper is the following theorem.
\begin{thm*}{\rm\textbf{\ref{202}}}
Let $\mathcal{X} \rightarrow S$ be a separated, locally finitely-presented morphism from a Deligne-Mumford stack $\mathcal{X}$ to an algebraic space $S$. Fix a line bundle $\mathcal{L}$ on $\mathcal{X}$. The moduli problem $\mathcal{M}_{\mathcal{X},\mathcal{L},Higgs}$ is represented by an algebraic space which is separated and locally finitely presented.
\end{thm*}
To prove the main result, we use a theorem by Artin in 1968 \cite[Theorem 5.3]{Art} (see Theorem \ref{301} in \S 3). In fact, Olsson and Starr also used this theorem to prove that the quotient functor is representable by an algebraic space \cite{OlSt}. Artin's theorem is the key point to prove Theorem \ref{202}.

In \S 2, we review some basic definitions and properties about Deligne-Mumford stack, define the moduli problem for Hitchin pairs and state our main result Theorem \ref{202}. \S 3 is devoted to the proof of Theorem \ref{202}. In \S 3.1, we review Artin's theorem \cite[Theorem 5.3]{Art} and give a brief overview of the proof. In \S 3.2 and \S 3.3, we prove the main theorem in this paper. Moreover, we construct a well-defined deformation and construction for the moduli problem $\mathcal{M}_{\mathcal{X},\mathcal{L},Higgs}$, which is a generalization of the infinitesimal deformation of Higgs bundle \cite{BisRam}.

\vspace{2mm}
\textbf{Acknowledgments}.
The author would like to thank Georgios Kydonakis and Lutian Zhao for helpful discussions.
\vspace{2mm}

\section{Definition and Statement of the Main Result}
\subsection{Definition}
In this section, we review the definition of Deligne-Mumford stack and coherent sheaves over an algebraic stack. Details can be found in \cite{Ol}. Let $S_0$ be a scheme, which is the spectrum of an algebraically closed field. Let $\text{Sch}/S_0$ be the category of schemes over $S_0$. Let $f: E \rightarrow F$ be a morphism of sheaves on $\text{Sch}/S_0$ with respect to the \'etale topology. The morphism $f$ is \emph{representable} by schemes if for every $S_0$-scheme $T$ and morphism $T \rightarrow F$, the fiber product $E \times_F T$ is a scheme. An \emph{algebraic space} $\mathcal{X}$ over $S_0$ is a functor $\mathcal{X}:(\text{Sch}/S_0)^{{\rm op}} \rightarrow \text{Set}$ such that
\begin{enumerate}
\item $\mathcal{X}$ is a sheaf with repsect to the \'etale topology.
\item $\Delta: \mathcal{X} \rightarrow \mathcal{X} \times_{S_0} \mathcal{X}$ is representable by schemes.
\item There exists an $S_0$-scheme $U$ and a surjective \'etale morphism $U \rightarrow \mathcal{X}$.
\end{enumerate} .

Sheaves on an algebraic space $\mathcal{X}$ is defined locally with respect to the \'etale topology of $\text{Sch}/S_0$. Consider a local chart $U \rightarrow \mathcal{X}$, an \'etale morphism with $U$ a scheme. Let $R_U=U \times_{\mathcal{X}} U$, which is a scheme by the representability of the diagonal morphism. Denote by $s,t: R_U \rightarrow U$ the source and target map. Let $(R_U \rightrightarrows U)_{\text{\'et}}$ be the set of pairs $(F_U, \epsilon_U)$, where $F_U$ is a sheaf on $U$ and
\begin{align*}
\epsilon_U: s^* F_U \rightarrow t^* F_U
\end{align*}
is an isomorphism of sheaves on $R_U$. Now let $U,V$ be two schemes, and let
\begin{align*}
f_U: U \rightarrow \mathcal{X}, \quad f_V : V \rightarrow \mathcal{X}
\end{align*}
be two \'etale morphisms. Let $h:U \rightarrow V$ be a morphism of schemes such that the following diagram commutes
\begin{center}
\begin{tikzcd}
U \arrow[rd, "f_U"] \arrow[rr, "h"] &  & V \arrow[ld,"f_V"] \\
& \mathcal{X} &
\end{tikzcd}
\end{center}
A \emph{sheaf} $\mathcal{F}$ on $\mathcal{X}$ consists of a collection of pairs $\{(F_U,\epsilon_U)\}$ for each \'etale morphism $f_U: U \rightarrow \mathcal{X}$ along with isomorphisms $a_{h}:F_U \rightarrow h^* F_V$, where $F_U$, $F_V$ are sheaves over $U$, $V$ respectively.

A sheaf $F$ is \emph{coherent} (resp. \emph{quasi-coherent}, \emph{locally free}) \emph{sheaf} if the local data $F_U$ is coherent (resp. quasi-coherent, locally free) for each \'etale morphism $U \rightarrow \mathcal{X}$. A sheaf $F=\{(F_U,\epsilon_U)\}$ is an \emph{$\mathcal{O}_{\mathcal{X}}$-module}, if $F_U$ is an $\mathcal{O}_U$-module for each \'etale morphism $U \rightarrow \mathcal{X}$.

It is easy to check that the category of sheaves over an algebraic space $\mathcal{X}$ is equivalent to the category of pairs $(F_U, \epsilon_U)$, where $U$ is a scheme and $U \rightarrow \mathcal{X}$ is a surjective \'etale morphism. Such a surjective \'etale morphism always exists by the definition of algebraic space. Under this definition, the category of sheaves over $\mathcal{X}$ is independent of the choice of the surjective \'etale morphism $U \rightarrow \mathcal{X}$ (see \cite[\S 7]{Ol}).

A \emph{stack} $\mathcal{X}: (\text{Sch}/S_0)^{{\rm op}} \rightarrow \text{Set}$ is a functor with groupoid as fiber (fibered category) satisfying the descent condition (\cite[\S 4.6]{Ol}). A morphism of stacks $f: \mathcal{X} \rightarrow \mathfrak{Y}$ is \emph{representable} if for every $S$-scheme $T$ and morphism $T \rightarrow \mathfrak{Y}$, the fiber product
\begin{align*}
\mathcal{X} \times_{\mathfrak{Y}} T
\end{align*}
is an algebraic space. An \emph{algebraic stack} $\mathcal{X}$ over scheme $S_0$ is a stack satisfying the following conditions:
\begin{enumerate}
\item The diagonal $$\Delta: \mathcal{X} \rightarrow \mathcal{X} \times_{S_0} \mathcal{X}$$ is representable.
\item There exists a smooth surjective morphism $U \rightarrow \mathcal{X}$ with $U$ a scheme.
\end{enumerate}
Moreover, if there exists a \'etale surjection $U \rightarrow \mathcal{X}$, the algebraic stack $\mathcal{X}$ is a \emph{Deligne-Mumford stack}.

\subsection{Moduli Space of Twisted Hitchin Pairs over Deligne-Mumform Stack}
In this section, we give the definition of the moduli problem of twisted Hitchin pairs over Deligne-Mumford stack and state the main result, Theorem \ref{202} of this paper.

Let $S$ be an algebraic space, which is locally of finite type over $S_0$, where $S_0$ is the spectrum of an algebraically closed field as we defined in \S 2.1. Locally, the algebraic space $S$ is a scheme of finite type over $\mathbb{Z}$ with respect to the \'etale topology. Denote by $(\text{Sch}/S)_{\text{\'et}}$ the category of $S$-schemes with respect to the \'etale topology. In other words, the objects in $(\text{Sch}/S)_{\text{\'et}}$ are \'etale morphism from schemes to $S$. In the rest of the paper, we use the notation $\text{Sch}/S$ instead of $(\text{Sch}/S)_{\text{\'et}}$. Let $\mathcal{X} \rightarrow S$ be a separated, locally finitely-presented morphism from a Deligne-Mumford stack $\mathcal{X}$ to $S$. Let $\mathcal{G}$ be a coherent $\mathcal{O}_{\mathcal{X}}$-module.
We define the functor
\begin{align*}
Q(\mathcal{G}/\mathcal{X}/S):=Quot(\mathcal{G},\mathcal{X},S): (\text{Sch}/S)^{{\rm op}} \rightarrow \text{Set}
\end{align*}
as follows. For each $S$-scheme $T$, define $\mathcal{X}_T$ as $\mathcal{X} \times_{S} T$ and $\mathcal{G}_T$ the pullback of $\mathcal{G}$ to $\mathcal{X}_T$. Define $Q(\mathcal{G}/\mathcal{X}/S)(T)$ to be the set of $\mathcal{O}_{\mathcal{X}_T}$-module quotients $\mathcal{G}_T \rightarrow \mathcal{F}_T$ such that
\begin{enumerate}
\item $\mathcal{F}_T$ is a locally finitely-presneted quasi-coherent $\mathcal{O}_{\mathcal{X}_T}$-module;
\item $\mathcal{F}_T$ is flat over $T$;
\item the support of $\mathcal{F}_T$ is proper over $T$.
\end{enumerate}
The functor $Q(\mathcal{G}/\mathcal{X}/S)$ is called the \emph{quotient functor}.

Artin proved that the quotient functor $Q(\mathcal{G}/\mathcal{X}/S)$ is represented by a separated and locally finitely-presented algebraic space over $S$ when $\mathcal{X}$ is an algebraic space \cite{Art}. Olsson and Starr generalized this result to Deligne-Mumford stack  \cite{OlSt}.

\begin{thm}[Theorem 1.1 in \cite{OlSt}]\label{201}
With respect to the above notation, the functor $Q(\mathcal{G}/\mathcal{X}/S)=Quot(\mathcal{G},\mathcal{X},S)$ is represented by an algebraic space which is separated and locally finitely presented over $S$.
\end{thm}

Let $\mathcal{F}$ be a locally finitely-presented quasi-coherent $\mathcal{O}_{\mathcal{X}}$-module in $Q(\mathcal{G}/\mathcal{X}/S)$, and we fix a line bundle (locally free sheaf with rank one) $\mathcal{L}$ over $\mathcal{X}$, which is considered as the twisted bundle. An \emph{$\mathcal{L}$-twisted Higgs field} $\Phi$ on the quasi-coherent sheaf $\mathcal{F}$ is a homomorphism
\begin{align*}
\Phi:\mathcal{F} \rightarrow \mathcal{F} \otimes \mathcal{L}.
\end{align*}
An \emph{$\mathcal{L}$-twisted Hitchin pair} over $\mathcal{X}$ is a pair $(\mathcal{F},\Phi)$, where $\mathcal{F}$ is a locally finitely-presented quasi-coherent sheaf over $\mathcal{X}$ and $\Phi$ is an $\mathcal{L}$-twisted Higgs field. We consider the following moduli problem (functor)
\begin{align*}
\mathcal{M}_{\mathcal{X},\mathcal{L},Higgs}: (\text{Sch/S})^{{\rm op}} \rightarrow \text{Set}
\end{align*}
for $\mathcal{L}$-twisted Hitchin pairs over $\mathcal{X}$ as follows. For each $T \in \text{Sch/S}$, define
\begin{align*}
\mathcal{M}_{\mathcal{X},\mathcal{L},Higgs}(T)=\{ (\mathcal{F}_T,\Phi_T) \text{ } | \text{ }\mathcal{F}_T \in Q(\mathcal{G}/\mathcal{X}/S)(T), \Phi_T: \mathcal{F}_T \rightarrow \mathcal{F}_T \otimes p_{\mathcal{X}}^* \mathcal{L} \} ,
\end{align*}
where $p_{\mathcal{X}}:\mathcal{X} \times_S T \rightarrow \mathcal{X}$ is the natural projection.

Note that the definition of the moduli problem $\mathcal{M}_{\mathcal{X},\mathcal{L},Higgs}$ depends on the choice of the quotient functor $Q(\mathcal{G}/\mathcal{X}/S)$. But, this definition can be extended to the category, more precisely the stack, of coherent sheaves $\mathfrak{Coh}_{\mathcal{X}}$ over $\mathcal{X}$. Nironi proved that a certain family of quotient functors form a smooth open atlas of the coherent sheaves $\mathfrak{Coh}_{\mathcal{X}}$ \cite[\S 2]{Nir}. Thus the moduli problem $\mathcal{M}_{\mathcal{X},\mathcal{L},Higgs}$ can be naturally extended to the stack of coherent sheaves $\mathfrak{Coh}_{\mathcal{X}}$ over $\mathcal{X}$. Based on this property, we neglect the quotient functor in the definition of $\mathcal{M}_{\mathcal{X},\mathcal{L},Higgs}$. Now we are ready to state the main theorem in this paper.
\begin{thm}\label{202}
Let $\mathcal{X} \rightarrow S$ be a separated, locally finitely-presented morphism from a Deligne-Mumford stack $\mathcal{X}$ to an algebraic space $S$. Fix a line bundle $\mathcal{L}$ on $\mathcal{X}$. The moduli problem $\mathcal{M}_{\mathcal{X},\mathcal{L},Higgs}$ is represented by an algebraic space which is separated and locally finitely presented.
\end{thm}
This theorem will be proved in the next section.


\section{Proof of Theorem \ref{202}}
\subsection{Background and Overview of the Proof}
We first review the \emph{deformation theory} defined in \cite{Art} and give the statement of Theorem 5.3 in \cite{Art}, which is the key to prove Theorem \ref{202}. Then, we go over some basic ideas of the proof of Theorem \ref{201} that the quotient functor $Q(\mathcal{G}/\mathcal{X}/S)$ is represented by a separated and locally finitely presented algebraic space \cite[Theorem 1.1]{OlSt}. Finally, we give a brief overview of the proof of Theorem \ref{202} at the end of this subsection.

Let $S$ be an algebraic space. An \emph{infinitesimal extension} of an $\mathcal{O}_{S}$-algebra $A$ is a surjective map of $\mathcal{O}_{S}$-algebras $A' \twoheadrightarrow A$ such that the kernel $M=\text{ker}(A' \rightarrow A)$ is a finitely generated nilpotent ideal.

Let $F$ be a contravariant functor from $\text{Sch}/S$ to sets (moduli problem). Let $A_0$ be a noetherian $\mathcal{O}_{S}$-domain. We prefer to use the notation $F(A_0)$ instead of $F(\Spec A_0)$. The \emph{deformation situation} is defined as a triple
\begin{align*}
(A' \rightarrow A \rightarrow A_0, M, \xi)
\end{align*}
where $A' \rightarrow A \rightarrow A_0$ is a diagram of infinitesimal extension, $M=\text{ker}(A' \rightarrow A)$ a finite $A_0$-module and $\xi \in F(A_0)$. Let $\xi$ be an element in $F(A_0)$. As a contravariant functor (for $S$-schemes), we have a natural map $F(A) \rightarrow F(A_0)$. Denote by $F_{\xi}(A)$ the set of elements in $F(A)$ whose image is $\xi \in F(A_0)$.

The \emph{deformation theory} we consider in this paper is described in \cite[Definition 5.2]{Art}. A \emph{deformation theory} for $F$ consists of the following data and conditions
\begin{enumerate}
\item A functor associates to every triple $(A_0,M,\xi)$ an $A_0$-module $D=D(A_0,M,\xi)$, and to every map of triples $(A_0,M,\xi) \rightarrow (B_0,N,\eta)$ an linear map $D(A_0,M,\xi)\rightarrow D(B_0,N,\eta)$.
\item For every deformation situation, there is an operation of the additive group of $D(A_0,M,\xi)$ on $F_{\xi}(A')$ such that two elements are in the same orbit under the operation if and only if they have the same image in $F_{\xi}(A)$, where $F_{\xi}(A')$ is the subset of $F(A')$ of elements whose image in $F(A_0)$ is $\xi$.
\end{enumerate}

\begin{thm}[Theorem 5.3 in \cite{Art}]\label{301}
Let $F$ be a functor on $({\rm Sch/S})^{\rm op}$. Given a deformation theory for $F$, then $F$ is represented by a separated and locally of finite type algebraic space over $S$, if the following conditions hold:
\begin{enumerate}
\item[{\rm (1)}] $F$ is a sheaf for the $fppf$-topology and $F$ is locally of finite presentatiion.
\item[{\rm (2)}]{\rm (Inverse Limits)} Let $\bar{A}$ be a complete noetherian local $\mathcal{O}_S$-algebra with residue field of finite type over $S$ and let $\mathfrak{m}$ be the maximal idea of $\bar{A}$. Then the canonical map $F(\bar{A}) \rightarrow \lim\limits_{\leftarrow} F(\bar{A}/\mathfrak{m}^n)$ is injective, and its image is dense in $\lim\limits_{\leftarrow} F(\bar{A}\mathfrak{/m}^n)$.
\item[{\rm (3)}]{\rm (Seperation)}
\begin{enumerate}
\item[{\rm (a)}] Let $A_0$ be a \emph{geometric discrete valuation ring}, which is a localization of a finite type $\mathcal{O}_S$-algebra with residue field of finite type over $\mathcal{O}_S$. Let $K,k$ be its fraction field and residue field respectively. If $\xi,\eta \in F(A_0)$ induce the same element in $F(K)$ and $F(k)$, then $\xi=\eta$.
\item[{\rm (b)}] Let $A_0$ be an $\mathcal{O}_S$-integral domain of finite type. Let $\xi,\eta \in F(A_0)$. Suppose that there is a dense set $\mathcal{S}$ in $\text{Spec}(A_0)$ such that $\xi=\eta$ in $F(k(s))$ for all $s \in \mathcal{S}$. Then $\xi=\eta$ on a non-empty open subset of $\text{Spec}(A_0)$.
\end{enumerate}
\item[{\rm (4)}]{\rm (Deformation)}
\begin{enumerate}
\item[{\rm (a)}] The module $D=D(A_0,M,\xi)$ commutes with localization in $A_0$ and is a finite module when $M$ is free of rank one.
\item[{\rm (b)}] The module operates freely on $F_\xi(A')$ when $M$ is of length one.
\item[{\rm (c)}] Let $A_0$ be an $\mathcal{O}_S$-integral domain of finite type. There is a non-empty open set $U$ of $\text{Spec}(A_0)$ such that for every closed point $s \in U$, we have $$D \otimes_{A_0} k(s)=D(k,M\otimes_{A_0}k(s), \xi_s)$$.
\end{enumerate}

\item[{\rm (5)}]{\rm (Obstruction)} Suppose we have a deformation situation $(A' \rightarrow A \rightarrow A_0,M,\xi)$.
\begin{enumerate}
\item[{\rm (a)}] Let $A_0$ be of finite type and $M$ of length one. Let
\begin{center}
\begin{tikzcd}
B' \arrow[r] \arrow[d]  & B \arrow[r] \arrow[d] & A_0 \arrow[d] \\
A' \arrow[r] & A \arrow[r] & A_0
\end{tikzcd}
\end{center}
be a diagram of infinitesimal extensions of $A_0$ with $B'=A' \times_A B$. If $b \in F(B)$ is an element lying over $\xi$ whose image $a \in F(A)$ can be lifted to $F(A')$, then $b$ can be lifted to $F(B')$.
\item[{\rm (b)}] $A_0$ is a geometric discrete valuation ring with fraction field $K$ and $M$ free of rank one. Denote by $A_K, A'_K$ the localizations of $A,A'$ respectively. If the image of $\xi$ in $F(A_K)$ can be lifted to $F(A'_K)$, then its image in $F(A_0 \times_K A_K)$ can be lifted to $F(A_0 \times_K A'_K)$.

\item[{\rm (c)}] With the same notation in 5(b). Let $M$ be a free module of rank $n$ and $\xi \in F(A)$. Suppose that for every one-dimensional quotient $M^*_K$ of $M_K$ the lifting of $\xi_K$ to $F(A^*_K)$ is obstructed, where $A'_K \rightarrow A^*_K \rightarrow A_K$ is the extension determined by $M^*_K$. Then there is a non-empty open set $U$ of ${\rm Spec}(A_0)$ such that for every quotient $\epsilon: M  \rightarrow M^*$ of length one with support in $U$,  the lifting of $\xi$ to $F(A^*)$ is obstructed, where $A' \rightarrow A^* \rightarrow A$ denotes the resulting extension.
\end{enumerate}

\end{enumerate}
\end{thm}

Olsson and Starr used Theorem \ref{301} to prove that the quotient functor $Q(\mathcal{G}/\mathcal{X}/S)$ is represented by a separated and locally finitely presented algebraic space. In other words, the quotient functor $Q(\mathcal{G}/\mathcal{X}/S)$ satisfies conditions (1) to (5) in Theorem \ref{301}.

Now we go back to the moduli problem for $\mathcal{L}$-twisted Higgs bundle $\mathcal{M}_{\mathcal{X},\mathcal{L},Higgs}$. In this section, we use $\mathcal{M}:=\mathcal{M}_{\mathcal{X},\mathcal{L},Higgs}$ to simplify the notation for the moduli problem we are interested in. Note that the problem of representability of $\mathcal{M}$ is \'etale local on $S$. Therefore we may assume that $S$ is an affine scheme and of finite type over ${\rm Spec}(\mathbb{Z})$ as we explained at the beginning of \S 2.2. It is easy to check that the functor $\mathcal{M}$ satisfies the first condition. We will prove conditions $(2)$ and $(3)$ in \S 3.2. The deformation and obstruction theory will be discussed in \S 3.3. Condition $(4)$ will be proved in \S 3.3.1 by constructing a well-defined deformation theory for $\mathcal{M}$. The obstruction property will be discussed in \S 3.3.2.

\subsection{Inverse Limit and Separation}
We prove the inverse limit condition and separation condition in this subsection. Olsson and Starr proved that the quotient functor $Q(\mathcal{G}/\mathcal{X}/S)$ is represented by pa separated, locally finitely-presented algebraic space over $S$ \cite[Theorem 1.1]{OlSt}. In other words, the functor $Q(\mathcal{G}/\mathcal{X}/S)$ preserves the inverse limit and satisfies the separation condition. We use these properties of the quotient functor $Q(\mathcal{G}/\mathcal{X}/S)$ to prove the properties of \emph{inverse limit} and \emph{separation} of the functor $\mathcal{M}$.

To prove a functor $F$ satisfying the inverse limit condition, we have to prove that the map $F(\bar{A}) \rightarrow \lim\limits_{\leftarrow} F(\bar{A}/\mathfrak{m}^n)$ is injective and for any $(\xi_n) \in \lim\limits_{\leftarrow} F(\bar{A}/\mathfrak{m}^n)$ there is an element $\xi' \in F(\bar{A})$ which induces $\xi_1 \in F(\bar{A}/\mathfrak{m}^2)$. Let $\hat{\mathcal{X}}=\lim (\mathcal{X} \otimes_A A/\mathfrak{m}^{n})$. There is a natural morphism $j : \hat{\mathcal{X}} \rightarrow \mathcal{X}$. This morphism induces the following map
\begin{align*}
j: Q(\mathcal{G}/\mathcal{X}/S)(\bar{A}) \rightarrow \lim\limits_{\leftarrow} Q(\mathcal{G}/\mathcal{X}/S)(\bar{A}/\mathfrak{m}^n), \quad \mathcal{F} \rightarrow j^* \mathcal{F}.
\end{align*}
This map is injective and has a dense image by \cite[Theorem 1.1]{OlSt}. In other words, let $(\mathcal{F}_n)$ be an element in $\lim\limits_{\leftarrow} Q(\mathcal{G}/\mathcal{X}/S)(\bar{A}/\mathfrak{m}^n)$. There is an element $\mathcal{F}' \in Q(\mathcal{G}/\mathcal{X}/S)(\bar{A})$ such that $\mathcal{F}'$ induces $\mathcal{F}_1 \in Q(\mathcal{G}/\mathcal{X}/S)(\bar{A}/\mathfrak{m}^2)$. The same argument holds for the sheaf $\text{End}(\mathcal{F})$. Given an element $\text{End}(\mathcal{F}_n) \in \lim\limits_{\leftarrow} Q(\mathcal{G}/\mathcal{X}/S)(\bar{A}/\mathfrak{m}^n)$, there is an element $\text{End}(\mathcal{F}')  \in Q(\mathcal{G}/\mathcal{X}/S)(\bar{A})$ such that $\text{End}(\mathcal{F}')$ induces $\text{End}(\mathcal{F}_1) \in Q(\mathcal{G}/\mathcal{X}/S)(\bar{A}/\mathfrak{m}^2)$. Note that the morphism $j : \hat{\mathcal{X}} \rightarrow \mathcal{X}$ also induces the following map
\begin{align*}
j: \mathcal{M}(\bar{A}) \rightarrow \lim\limits_{\leftarrow} \mathcal{M}(\bar{A}/\mathfrak{m}^n), \quad (\mathcal{F},\Phi) \rightarrow (j^*\mathcal{F},j^* \Phi),
\end{align*}
where $\mathcal{F} \in Q(\mathcal{G}/\mathcal{X}/S)(\bar{A})$ and $\Phi \in \text{End}(\mathcal{F}) \otimes \mathcal{L}$. The map $j$ is clearly injective. Given any element $\left( (\mathcal{F}_n,\Phi_n) \right)_{n \geq 1} \in \lim\limits_{\leftarrow} \mathcal{M}(\bar{A}/\mathfrak{m}^n)$, where $\mathcal{F}_n \in Q(\mathcal{G}/\mathcal{X}/S)(\bar{A}/\mathfrak{m}^n)$ and $\Phi_n \in \text{End}(\mathcal{F}_n)\otimes \mathcal{L}$, there exists $\mathcal{F}' \in Q(\mathcal{G}/\mathcal{X}/S)(\bar{A})$ such that $\mathcal{F}'$ induces $\mathcal{F}_1$. Thus we can find an element $\Phi' \in \text{End}(\mathcal{F}') \otimes \mathcal{L}$ such that $\Phi'$ induces $\Phi_1$. Therefore the map $j: \mathcal{M}(\bar{A}) \rightarrow \lim\limits_{\leftarrow} \mathcal{M}(\bar{A}/\mathfrak{m}^n)$ has a dense image and the functor $\mathcal{M}$ satisfies the inverse limit condition.

Now we are going to prove the separation property. With the same notations as in Theorem \ref{301} (3), let $\xi=(\mathcal{F}_{\xi},\Phi_{\xi})$ and $\eta=(\mathcal{F}_{\eta},\Phi_{\eta})$ be two elements in $\mathcal{M}(A_0)$ inducing the same element in $\mathcal{M}(K)$ and $\mathcal{M}(k)$. By Theorem \ref{201}, the quotient functor $Q(\mathcal{G}/\mathcal{X}/S)$ satisfies the separation condition. Thus we have $\mathcal{F}_{\xi}=\mathcal{F}_{\eta}$, which also implies that $\text{End}(\mathcal{F}_{\xi}) \otimes \mathcal{L}=\text{End}(\mathcal{F}_{\eta}) \otimes \mathcal{L}$. Therefore we have $\xi=\eta$. The same argument works for the condition ${\rm (3b)}$. This finishes the proof of the separation condition.

\subsection{Deformation and Obstruction Theory}
With respect to the same notation as in \S 3.1, let $\xi=(\mathcal{F},\Phi) \in \mathcal{M}(A_0)$. The key object we will study in \S 3.3 about the deformation and obstruction theory is $\mathcal{M}_{\xi}(A_0[M])$, where $\mathcal{M}_{\xi}(A_0[M])$ is the set of elements in $\mathcal{M}(A_0[M])$ whose restriction to $A_0$ is $\xi$. The multiplication of the ring $A_0[M]$ is defined as
\begin{align*}
(a,m)(a',m')=(aa',am'+a'm).
\end{align*}
Clearly, $M$ is a nilpotent ideal in $A_0[M]$.

Note that the conditions of the deformation theory and obstruction theory ($(4)$ and $(5)$ in Theorem \ref{301}) are local property on both $\mathcal{X}$ and $S$. Thus we can take a local chart $U_{\mathcal{X}} \rightarrow \mathcal{X}$ of $\mathcal{X}$ and a local chart $U_S \rightarrow S$ of $S$, where $U_{\mathcal{X}}$ and $U_{\mathcal{S}}$ are schemes. We have the following diagram
\begin{center}
\begin{tikzcd}
U_{\mathcal{X}} \times_S U_S   \arrow[r] \arrow[d]  & \mathcal{X} \times_S U_S \arrow[r] \arrow[d] & U_S \arrow[d] \\
U_{\mathcal{X}} \arrow[r] & \mathcal{X} \arrow[r] & S
\end{tikzcd}
\end{center}
Thus we work on the separated and locally finitely-presented morphism $U_{\mathcal{X}} \times_S U_S \rightarrow U_S$ locally. By the definition of Deligne-Mumford stack, $U_{\mathcal{X}} \times_S U_S$ is an algebraic space. Thus we may assume that $S$ is an affine scheme and $\mathcal{X}$ is an algebraic space when we do the deformation (\S 3.3.1) and obstruction theory (\S 3.3.2).

At the end of the setup of the deformation and obstruction theory, we review the statement of the well-known \emph{five lemma}, which will be used frequently in this subsection.
\begin{lem}[Five Lemma]
Assume that all objects below are in an abelian category.
\begin{center}
\begin{tikzcd}
A  \arrow[r] \arrow[d,"a"]  & B \arrow[r] \arrow[d,"b"] & C \arrow[r] \arrow[d,"c"] & D \arrow[r] \arrow[d,"d"] & E \arrow[d,"e"]\\
A' \arrow[r] & B' \arrow[r] & C' \arrow[r] & D' \arrow[r] & E'
\end{tikzcd}
\end{center}
If the rows are exact, $b$ and $d$ are isomorphisms, $a$ is an epimorphism and $e$ is a monomorphism, then $c$ is an isomorphism.
\end{lem}

\subsubsection{Deformation Theory}
By the definition of the deformation theory in \S 3.1, we have to construct an $A_0$-module $D(A_0,M,\xi)$ for each triple $(A_0,M,\xi)$, where $\xi \in \mathcal{M}(A_0)$. In this section, we calculate the $A_0$-module $\mathcal{M}_{\xi}(A_0[M])$ and prove that this module is the correct deformation theory for $\mathcal{M}$. Biswas and Ramanan calculated the infinitesimal deformation theory for $\mathcal{M}$, i.e. $\mathcal{M}_{\xi}(\mathbb{C}[\varepsilon])$, $\xi \in \mathcal{M}(\mathbb{C})$ \cite{BisRam}. We generalize their approach to the deformation theory in this paper.

Let us consider a special case first. Let $A'=A_0[M]:=A_0 \oplus M$. Let $(\mathcal{F},\Phi)$ be an element in $\mathcal{M}(A_0)$. Define $\mathcal{F}'=\mathcal{F} \times_{\text{Spec }A_0} \text{Spec } A'$. Abusing the notation, we consider $\mathcal{F}'$ as $\mathcal{F}\oplus\mathcal{F}[M]$. For a section $s$ of $\text{End}(\mathcal{F}) \otimes M$, the corresponding automorphism of $\mathcal{F}'$ is denoted by $1+s$. Moreover, if $v+w$ is a section of $\text{End}(\mathcal{F}') \otimes \mathcal{L}'$, where $\mathcal{L}'$ is the pull-back of $\mathcal{L}$ under the projection $ \text{Spec } A' \rightarrow \text{Spec } A_0$, we have
\begin{align*}
\rho(1+s)(v+w)=v+w+\rho(s)(v),
\end{align*}
where $\rho$ is the natural action of $\text{End}(\mathcal{F})$ on itself. The deformation complex $C_M^{\bullet}(\mathcal{F},\Phi)$ is defined as follows
\begin{align*}
C_M^{\bullet}(\mathcal{F},\Phi): C_M^0(\mathcal{F}) = \text{End}(\mathcal{F}) \otimes M \xrightarrow{e(\Phi)} C_M^1(\mathcal{F}) = \text{End}(\mathcal{F}) \otimes \mathcal{L} \otimes M,
\end{align*}
where the map $e(\Phi)$ is given by
\begin{align*}
e(\Phi)(s)=-\rho(s)(\Phi).
\end{align*}
If there is no ambiguity, we omit the notations $M$, $\mathcal{F}$, $\Phi$ and use the following notation
\begin{align*}
C^{\bullet}: C^0 = \text{End}(\mathcal{F}) \otimes M \xrightarrow{e(\Phi)} C^1 = \text{End}(\mathcal{F}) \otimes \mathcal{L} \otimes M
\end{align*}
for the deformation complex.

Now we are ready to calculate $\mathcal{M}_{\xi}(A_0[M])$. The following proposition is a generalization of Theorem 2.3 in \cite{BisRam}.
\begin{prop}\label{302}
Let $\xi =(\mathcal{F},\Phi)$ be an $\mathcal{L}$-twisted Hitchin pair in $\mathcal{M}(A_0)$. The set $\mathcal{M}_{\xi}(A_0[M])$ is isomorphic to the hypercohomology group $\mathbb{H}^1(C^{\bullet})$, where $C^{\bullet}$ is the complex
\begin{align*}
C^{\bullet}:C^0 ={\rm End}(\mathcal{F}) \otimes M \xrightarrow{e(\Phi)} C^1 = {\rm End}(\mathcal{F}) \otimes \mathcal{L} \otimes M,
\end{align*}
where $e(\Phi)(s)=-\rho(s)(\Phi)$ is defined as above.
\end{prop}

\begin{proof}
Let $\mathcal{U}=\{U_i=\text{Spec}(A_i)\}$ be an \'etale covering of $\mathcal{X}$ by open affine schemes. The covering $\mathcal{U}$ of $\mathcal{X}$ also gives an \'etale covering $\{U_i \times \text{Spec}(A_0)\}$ of $\mathcal{X} \times \text{Spec}(A_0)$. To be precise, the product $U_i \times \text{Spec}(A_0)$ is taken over $S$, i.e. $U_i \times_{S} \text{Spec}(A_0)$. We omit the base scheme $S$ to simplify the notation. Define $U_i[M]=U_i \times \text{Spec}(A_0[M])$. Set
\begin{align*}
\text{End}(\mathcal{F})\otimes M|_{U_i[M]}=C^0_i, \quad \text{End}(\mathcal{F})\otimes \mathcal{L} \otimes M|_{U_i[M]}=C^1_{i},
\end{align*}
where $C^0_i$ and $C^1_i$ are $A_0$-modules. Similarly, modules $C^0_{ij}$ (resp. $C^1_{ij}$) are resctrictions of $C^0$ (resp. $C^1$) to $U_{ij}=U_i \bigcap U_j$. We consider the following $\hat{C}$ech resolution of $C^{\bullet}$:
\begin{center}
\begin{tikzcd}
& 0 \arrow[d] & 0 \arrow[d] & \\
0 \arrow[r] & C^0 \arrow[r,"e(\Phi)"] \arrow[d,"d^0_0"] & C^1 \arrow[r] \arrow[d,"d^1_0"] & 0 \\
0 \arrow[r] & \sum C_i^0 \arrow[r,"e(\Phi)"] \arrow[d,"d^0_1"] & \sum C_i^1 \arrow[r] \arrow[d,"d^1_1"] & 0 \\
0 \arrow[r] & \sum C_{ij}^0 \arrow[r,"e(\Phi)"] \arrow[d,"d^0_2"] & \sum C_{ij}^1 \arrow[r] \arrow[d,"d^1_2"] & 0 \\
& \vdots  & \vdots  &
\end{tikzcd}
\end{center}
We calculate the first hypercohomology $\mathbb{H}^1(C^{\bullet})$ from the above diagram. Let $Z$ be the set of pairs $(s_{ij}, t_i)$, where $s_{ij} \in C^0_{ij}$ and $t_i \in C^1_i$ satisfying the following conditions:
\begin{enumerate}
\item $s_{ij}+s_{jk}=s_{ik}$ as elements of $C^0_{ijk}$.
\item $t_i-t_j=e(\Phi)(s_{ij})$ as elements of $C^1_{ij}$.
\end{enumerate}
Let $B$ be the subset of $Z$ consisting of elements $(s_i-s_j,e(\Phi)(s_i))$, where $s_i \in C^0_i$. Clearly, $\mathbb{H}^1(C^{\bullet}) = Z/B$.

Given an element $(s_{ij},t_i) \in Z$, we shall construct a $\mathcal{L}$-twisted Higgs bundle $(\mathcal{F}',\Phi')$ on $\mathcal{X}\times \text{Spec}(A_0[M])$ such that $\mathcal{F}'|_{\mathcal{X}\times \text{Spec}(A_0)} \cong \mathcal{F}$ and $\Phi' |_{\mathcal{X}\times \text{Spec}(A_0)} \cong \Phi$.

For each $U_i[M]$, there is a natural projection $\pi: U_i[M] \rightarrow U_i \times \text{Spec}(A_0)$. Take the sheaf $\mathcal{F}'_i=\pi^*(\mathcal{F}|_{U_i \times \text{Spec}(A_0)})$. By the first condition of $Z$, we can identify the restrictions of $\mathcal{F}'_i$ and $\mathcal{F}'_j$ to $U_{ij}[M]$ by the isomorphism $1+s_{ij}$ of $\mathcal{F}'_{ij}$. Therefore we get a well-defined quasi-coherent sheaf $\mathcal{F}'$ on $\mathcal{X} \otimes \text{Spec}(A_0[M])$.

On each affine set $U_i[M]$, we have $\Phi_i+t_i : End(\mathcal{F}'_i) \otimes \mathcal{L}'$. It is easy to check $$e(\Phi_i + t_i)(1+s_{ij})=\Phi_j+t_j$$
by the second condition of $Z$. Therefore $\{\Phi_i+t_i\}$ can be glued together to give a global homomorphism $\Phi':\mathcal{F}' \rightarrow  \mathcal{F}' \otimes \mathcal{L}'$. For the element $(s_{ij},t_i)$ in $Z$, we construct an element $(\mathcal{F}',\Phi')$ in $\mathcal{M}_{\xi}(A_0[M])$.

Let $(s_{ij},t_i)$ be an element in $B$. In other words, $s_{ij}=s_i-s_j$ and $t_i=e(\Phi)(s_i)$. The identification of $\mathcal{F}'_i \cong \mathcal{F}'_j$ on $U_{ij}[M]$ is given by the isomorphism $$1+s_{ij}=1+(s_i-s_j).$$ Consider the following diagram
\begin{center}
\begin{tikzcd}
\mathcal{F}'_{ij} \arrow[d,"1+s_{ij}"] \arrow[r,"1+s_{i}"] & \mathcal{F}'_{ij} \arrow[d, "\text{Id}"]  \\
\mathcal{F}'_{ij} \arrow[r,"1+s_{j}"] & \mathcal{F}'_{ij}
\end{tikzcd}
\end{center}
The commutativity of the above diagram implies that $E'$ is trivial. Similarly, we have $$e(\Phi_i +t_i)(1+s_i)=\Phi_i.$$
Therefore the associated Hitchin pair $(\mathcal{F}',\Phi')$ is isomorphic to $(\pi^*\mathcal{F},\pi^* \Phi)$.

The above construction gives us a well-defined map from $\mathbb{H}^1(C^{\bullet})$ to $\mathcal{M}_{\xi}(A_0[M])$.

Now we have to construct the inverse map from $\mathcal{M}_{\xi}(A_0[M])$ to $\mathbb{H}^1(C^{\bullet})$. Let $(\mathcal{F}',\Phi') \in \mathcal{M}_{\xi}(A_0[M])$ be a $\mathcal{L}$-twisted Hitchin pairs over $\mathcal{X} \times \text{Spec}(A_0[M])$ such that $(\mathcal{F}'|_{\mathcal{X} \times \text{Spec}(A_0)},\Phi'|_{\mathcal{X} \times \text{Spec}(A_0)})=\xi=(\mathcal{F},\Phi)$. We still use the covering $\{U_i[M]\}$ of $\mathcal{X} \times \text{Spec}(A_0[M])$ to work on this problem locally.

Clearly, $\mathcal{F}'_i=\mathcal{F}'|_{U_i[M]}$ is the pull-back of $\mathcal{F}|_{\mathcal{X} \times \text{Spec}(A_0)}$. We obtain $\mathcal{F}'$ by gluing $\mathcal{F}'_i$ together. Thus the autormorphism $1+s_{ij}$ of $\mathcal{F}'_{ij}$ over the intersection $U_{ij}[M]$ should satisfy the condition $s_{ij}+s_{jk}=s_{ik}$ on $U_{ijk}[M]$. Similarly, $\Phi'$ is given by $\Phi_i+t_i$ locally, where $\Phi_i \in \text{End}(\mathcal{F}) \otimes \mathcal{L}|_{U_i}$ and $t_i \in \text{End}(\mathcal{F})\otimes \mathcal{L} \otimes M|_{U_i[M]}$. By the compatbility condition of $\Phi_i+t_i$ on $U_{ij}[M]$, we have $$e(\Phi_i+t_i)(1+s_{ij})=\Phi_j+t_j,$$
which gives us $e(\Phi)(s_{ij})=t_i-t_j$. Therefore, $(s_{ij},t_i) \in Z$.

The above discussion gives us a map from $\mathcal{M}_{\xi}(A_0[M])$ to $\mathbb{H}^1(C^{\bullet})$. It is easy to check that these two maps are inverse to each other. We finish the proof of this proposition.
\end{proof}

\begin{cor}\label{303}
We have the following long exact sequence
\begin{align*}
0 & \rightarrow \mathbb{H}^0(C^{\bullet}) \rightarrow H^0(\mathcal{X},C^0) \rightarrow H^0(\mathcal{X},C^1) \\
 & \rightarrow \mathbb{H}^1(C^{\bullet}) \rightarrow H^1(\mathcal{X},C^0) \rightarrow H^1(\mathcal{X},C^1) \rightarrow \mathbb{H}^2(C^{\bullet}) \rightarrow \dots \quad .
\end{align*}
\end{cor}
\begin{proof}
This long exact sequence follows directly from the definition of hypercohomology (see \cite{BisRam}).
\end{proof}

\begin{cor}\label{304}
Let $0 \rightarrow M_1 \rightarrow M_2 \rightarrow M_3 \rightarrow 0$ be a short exact sequence for finitely generated $A_0$-modules. We have a long exact sequence for hypercohomology
\begin{align*}
\dots \rightarrow \mathbb{H}^i(C_{M_1}^{\bullet}) \rightarrow \mathbb{H}^i(C_{M_2}^{\bullet}) \rightarrow \mathbb{H}^i(C_{M_3}^{\bullet}) \rightarrow \mathbb{H}^{i+1}(C_{M_1}^{\bullet}) \rightarrow \dots \quad .
\end{align*}
\end{cor}

Now we are ready to check the conditions on the deformation theory $D=D(A_0,M,\xi)=\mathcal{M}_{\xi}(A_0[M])$ (condition (4) in Theorem \ref{301}).

Note that the cohomology $H^i(\mathcal{X},C^j)$ commutes with localization in $A_0$. Applying the well-known five lemma to the long exact sequence in Corollary \ref{303}, the module $\mathcal{M}_{\xi}(A_0[M]) \cong \mathbb{H}^1(C^{\bullet})$ also commutes with localization. Now let $M$ be a free $A_0$-module of rank one. We use the notation $A_0[\varepsilon]:=A_0[M]$, if $M$ is a free module with rank one. By the \emph{finiteness theorem} of cohomology over algebraic spaces \cite[\S 7.5]{Ol}, the modules $H^i(\mathcal{X},C^j)$ are finitely generated for $0 \leq i,j \leq 1$. Thus $\mathcal{M}_{\xi}(A_0[\varepsilon])$ is a finitely generated module by the long exact sequence in Corollary \ref{303}.

For condition ${\rm (4b)}$, it is enough to check the case $A=A_0$ and $A'=A_0[M]$. In other words, we have to define an action $D=\mathcal{M}_{\xi}(A_0[M])$ on itself and show that this action is free. By Proposition \ref{302}, we know that
\begin{align*}
\mathcal{M}_{\xi}(A_0[M]) \cong \mathbb{H}^1(C^{\bullet}) = Z/B,
\end{align*}
where $Z$ is the set of pairs $(s_{ij}, t_i)$, where $s_{ij} \in C^0_{ij}$ and $t_i \in C^1_i$ satisfying the following conditions
\begin{enumerate}
\item $s_{ij}+s_{jk}=s_{ik}$ as elements of $C^0_{ijk}$;
\item $t_i-t_j=e(\Phi)(s_{ij})$ as elements of $C^1_{ij}$,
\end{enumerate}
and $B$ is the subset of $Z$ consisting of elements $(s_i-s_j,e(\Phi)(s_i))$, where $s_i \in C^0_i$. There is a natural action of $Z$ on itself
\begin{align*}
(s'_{ij},t'_i)(s_{ij},t_i):=(s'_{ij}+s_{ij},t'_i+t_i),
\end{align*}
where $(s'_{ij},t'_i)$, $(s_{ij},t_i) \in Z$. This action can be naturally extended to a well-defined action of $Z/B$ on itself. Thus we define an action $D=\mathcal{M}_{\xi}(A_0[M])$ on itself. It is easy to check this action is free.

Now we will discuss the condition ${\rm (4c)}$. Note that the condition of ${\rm (4c)}$ is a local property. We may assume that $\mathcal{X}$ is an algebraic space and $S$ is an affine scheme as we mentioned as the beginning of \S 3.3. To prove ${\rm (4c)}$, we need the following lemma.
\begin{lem}[Lemma 6.8, 6.9 in \cite{Art}]\label{305}
Let $\mathcal{X}$ be an algebraic space of finite type over an affine scheme $S=\Spec B$, where $B$ is an integral domain. Let $\mathcal{F}$, $\mathcal{G}$ be two coherent sheaves on $\mathcal{X}$. Then there is a non-empty open set $S'$ of $S$ such that for each $s \in S'$, the canonical map is an isomorphism
\begin{align*}
{\rm \mathcal{E}xt}_{X}^{q}(\mathcal{F},\mathcal{G})_s \xrightarrow{\cong} {\rm \mathcal{E}xt}_{X_s}^q(\mathcal{F}_s,\mathcal{G}_s), \quad q \geq 0,
\end{align*}
and
\begin{align*}
H^q(\mathcal{X},\mathcal{F})\otimes_{A_0} k(s) \xrightarrow{\cong} H^q(\mathcal{X}_s,\mathcal{F}_s), \quad q \geq 0.
\end{align*}
\end{lem}

By the above lemma, we can find a non-empty open set $S'$ such that
\begin{align*}
{\rm \mathcal{E}xt}_{\mathcal{X}}^{q}(\mathcal{F},\mathcal{G})_s \xrightarrow{\cong} {\rm \mathcal{E}xt}_{\mathcal{X}_s}^q(\mathcal{F}_s,\mathcal{G}_s),
\end{align*}
and
\begin{align*}
H^i(\mathcal{X},C^j)_s \cong H^i(\mathcal{X}_s,C_s^j), \quad i \geq 0, j=1,2,
\end{align*}
for $s \in S'$. Thus we have
\begin{align*}
\mathbb{H}^1(C^{\bullet})_s \cong \mathbb{H}^1(C^{\bullet}_s)
\end{align*}
by applying five lemma to the long exact sequence in Lemma \ref{303}, where $C^{\bullet}_s$ is the restriction of the complex $C^{\bullet}$ to the point $s$. We finish the proof of the condition ${\rm (4c)}$.

\subsubsection{Obstruction}
Fix a deformation situation $(A'\rightarrow A \rightarrow A_0,M,\xi)$, where $M$ is a free $A_0$-module of rank $n$. For any quotient $\epsilon:M \rightarrow M^*$, let $A' \rightarrow A^*$ be the quotient of $A'$ defined by $M^*$.
\begin{center}
\begin{tikzcd}
M'  \arrow[r] \arrow[d, "\epsilon"] & A' \arrow[r] \arrow[d] & A \arrow[d] \\
M^* \arrow[r] & A^* \arrow[r] & A
\end{tikzcd}
\end{center}
We can define the deformation situation $(A^* \rightarrow A \rightarrow A_0, M^*,\xi)$. For any element $(\mathcal{F}^*,\Phi^*) \in \mathcal{M}_{\xi}(A^*)$, we want to lift it to a well-defined element in $\mathcal{M}_{\xi}(A')$. We claim that the obstruction for this lifting comes from the vanishing of the second hypercohomology group $\mathbb{H}^2(C_{ {\rm ker}\epsilon }^{\bullet})$. By Corollary \ref{304}, there is a long exact sequence for the hypercohomology groups
\begin{align*}
\dots \rightarrow \mathbb{H}^1(C_{ {\rm ker}\epsilon}^{\bullet}) \rightarrow \mathbb{H}^1(C_{M}^{\bullet}) \rightarrow \mathbb{H}^1(C_{M^*}^{\bullet}) \rightarrow \mathbb{H}^2(C_{{\rm ker}\epsilon}^{\bullet}) \rightarrow \dots \quad .
\end{align*}
Such a lifting exists if and only if the morphism $\mathbb{H}^1(C_{M}^{\bullet}) \rightarrow \mathbb{H}^1(C_{M^*}^{\bullet})$ is surjective. Thus the vanishing of the second hypercohomology $\mathbb{H}^2(C_{{\rm ker}\epsilon}^{\bullet})$ is necessary and sufficient for the existence of the lifting.

For {\rm (5a)}, we have to show that there exists a lifting $(\mathcal{F}_{B'},\Phi_{B'})$ of $(\mathcal{F}_B,\Phi_B)$ with respect to the lifting $(\mathcal{F}_{A'},\Phi_{A'})$ of $(\mathcal{F}_A,\Phi_A)$, where $B'=B \times_A A'$. If we forget the morphism part $\Phi$ and only consider the coherent sheaf $\mathcal{F}$, such a lifting $\mathcal{F}_{B'}$ exists by \cite{Art,OlSt}. More generally, the map
\begin{align*}
Q(B') \rightarrow Q(A') \times_{Q(A)}  Q(B)
\end{align*}
is bijective, where $Q(A):=Q(\mathcal{G}_A/\mathcal{X}_A/S)$ and so are $Q(A')$ and $Q(B)$ \cite{OlSt}. We do the same argument for $\text{End}(\mathcal{F})$. The lifting $\text{End}(\mathcal{F}_{B'})$ of $\text{End}(\mathcal{F}_B)$ exists by the same reason. Thus the lifting $\Phi_{B'}$ of $\Phi_B$ also exists.

The proof of {\rm (5b)} is similar to \cite[(5b') page 65]{Art}. There is an inclusion map $i:\mathcal{X}_{A'_K} \rightarrow X_{A_0 \times_K A'_K}$, which induces the isomorphism
\begin{align*}
\text{Ext}^q_{A'_K}(i^* \mathcal{F}_1,\mathcal{F}_2) \cong \text{Ext}^q_{A_0 \times_K A'_K}(\mathcal{F}_1,i_* \mathcal{F}_2), \quad q \geq 0,
\end{align*}
for coherent sheaves $\mathcal{F}_1$ and $\mathcal{F}_2$. Then we have $H^q_{A'_K}(\mathcal{X},\mathcal{F}) \cong H^q_{A_0 \times_K A'_K}(\mathcal{X},\mathcal{F})$ for coherent sheaf $\mathcal{F}$. Thus we have $\mathbb{H}^2_{A'_K}(C^{\bullet}_{M_K}) \cong \mathbb{H}^2_{A_0 \times_K A'_K}(C^{\bullet}_{M})$ by applying five lemma to the long exact sequence in Corollary \ref{303}. It follows that the obstruction to lift element to $\mathcal{M}(A_0 \times_K A'_K)$ and to $\mathcal{M}(A'_K)$ are the same.

The proof of {\rm (5c)} is similar to that of {\rm (4c)}. The difference is that {\rm (4c)} is working on the deformation $\mathbb{H}^1(C^{\bullet})$, while {\rm (5c)} is working on the obstruction $\mathbb{H}^2(C^{\bullet})$. We use the same notation as in the statement of {\rm (5c)}. Let $\xi \in \mathcal{M}(A)$. Suppost that for every one-dimensional quotient $M_K \rightarrow M^*_K$, there is a non-trivial obstruction to lift $\xi_K \in \mathcal{M}(A_K)$ to $\mathcal{M}(A^*_K)$, where $A^*_K$ is the extension defined by $M_K^*$. We have to prove that there exists an open subset $S' \subseteq S$ such that $\xi$ cannot be lifted to $\mathcal{M}(A^*)$.

By the discussion at the beginning of \S 3.3.2, we know that the vanishing of the second hypercohomology $\mathbb{H}^2(C_{N}^{\bullet})$ is necessary and sufficient for the lifting of $\xi \in \mathcal{M}(A)$ in $\mathcal{M}(A^*)$, where $N$ is any submodule of $M$. Thus we have to show
\begin{align*}
\mathbb{H}^2(C_{N}^{\bullet})_s \cong \mathbb{H}^2((C_N^{\bullet})_s).
\end{align*}
This isomorphism implies that there is a non-trivial obstruction to lift $\xi_K$ to $\mathcal{M}(A^*_K)$ if and only there exists a non-trivial obstruction to lift $\xi$ to $\mathcal{M}(A^*)$. By Lemma \ref{305}, we can choose an open set $S'$ of $S$ such that
\begin{align*}
{\rm \mathcal{E}xt}^q_{\mathcal{X}_s}((\mathcal{F}_1)_s,(\mathcal{F}_2)_s) \cong {\rm \mathcal{E}xt}^q_{\mathcal{X}}(\mathcal{F}_1,\mathcal{F}_2)_s
\end{align*}
for $s \in S'$. By the spectral sequence relating local and global Ext functor, we have
\begin{align*}
{\rm Ext}^q_{\mathcal{X}_s}((\mathcal{F}_1)_s,(\mathcal{F}_2)_s) \cong {\rm Ext}^q_{\mathcal{X}}(\mathcal{F}_1,\mathcal{F}_2)_s
\end{align*}
for $s \in S'$. Taking $\mathcal{F}_1=\mathcal{O}_{\mathcal{X}}$ and $\mathcal{F}_2=C^0$, we have
\begin{align*}
H^q(\mathcal{X},C_N^0)_s \cong H^q(\mathcal{X}_s,(C_N^0)_s),
\end{align*}
for $s \in S'$. Similarly, we have $H^q(\mathcal{X},C_N^1)_s \cong H^q(\mathcal{X}_s,(C_N^1)_s)$. Thus the isomorphism $\mathbb{H}^2(C_{N}^{\bullet})_s \cong \mathbb{H}^2((C_N^{\bullet})_s)$ holds. This finishes the proof of {\rm (5c)}.

\bibliographystyle{alpha}
\nocite{*}
\bibliography{DMS}

\begin{thebibliography}{{Yok}95}

\bibitem[Art69]{Art}
M.~Artin.
\newblock Algebraization of formal moduli. {I}.
\newblock In {\em Global {A}nalysis ({P}apers in {H}onor of {K}. {K}odaira)},
  pages 21--71. Univ. Tokyo Press, Tokyo, 1969.

\bibitem[BR94]{BisRam}
I.~Biswas and S.~Ramanan.
\newblock An infinitesimal study of the moduli of {H}itchin pairs.
\newblock {\em J. London Math. Soc. (2)}, 49(2):219--231, 1994.

\bibitem[Hit87]{Hit}
N.~J. Hitchin.
\newblock The self-duality equations on a {R}iemann surface.
\newblock {\em Proc. London Math. Soc. (3)}, 55(1):59--126, 1987.

\bibitem[MS80]{MeSe}
V.~B. Mehta and C.~S. Seshadri.
\newblock Moduli of vector bundles on curves with parabolic structures.
\newblock {\em Math. Ann.}, 248(3):205--239, 1980.

\bibitem[{Nir}09]{Nir}
Fabio {Nironi}.
\newblock Moduli spaces of semistable sheaves on projective deligne-mumford
  stacks.
\newblock {\em arXiv preprint arXiv:0811.0949}, 2009.

\bibitem[Nit91]{Nit}
Nitin Nitsure.
\newblock Moduli space of semistable pairs on a curve.
\newblock {\em Proc. London Math. Soc. (3)}, 62(2):275--300, 1991.

\bibitem[Ols16]{Ol}
Martin Olsson.
\newblock {\em Algebraic spaces and stacks}, volume~62 of {\em American
  Mathematical Society Colloquium Publications}.
\newblock American Mathematical Society, Providence, RI, 2016.

\bibitem[OS03]{OlSt}
Martin Olsson and Jason Starr.
\newblock Quot functors for {D}eligne-{M}umford stacks.
\newblock {\em Comm. Algebra}, 31(8):4069--4096, 2003.
\newblock Special issue in honor of Steven L. Kleiman.

\bibitem[Sim94]{Simp2}
Carlos~T. Simpson.
\newblock Moduli of representations of the fundamental group of a smooth
  projective variety. {I}.
\newblock {\em Inst. Hautes \'{E}tudes Sci. Publ. Math.}, (79):47--129, 1994.

\bibitem[{Yok}93]{Yoko1}
K{\^o}ji {Yokogawa}.
\newblock Compactification of moduli of parabolic sheaves and moduli of
  parabolic {Higgs} sheaves.
\newblock {\em Journal of Mathematics of Kyoto University}, 33(2):451--504,
  1993.

\bibitem[{Yok}95]{Yoko2}
K{\^o}ji {Yokogawa}.
\newblock Infinitesimal deformation of parabolic {H}iggs sheaves.
\newblock {\em International Journal of Mathematics}, 6(1):125--148, 1995.

\end{thebibliography}

\bigskip
\noindent\small{\textsc{Department of Mathematics, Sun Yat-Sen University}\\
135 Xingang W Rd, BinJiang Lu, Haizhu Qu, Guangzhou Shi, Guangdong Sheng, China}\\
\emph{E-mail address}:  \texttt{sunh66@mail.sysu.edu.cn}

\end{document}